\documentclass[11pt]{article}

\usepackage[margin=1in]{geometry}
\usepackage{setspace}
\usepackage[english]{babel}
\usepackage{natbib}
\usepackage[hidelinks]{hyperref}
\usepackage{amsmath,amsthm,amssymb,amsfonts,amscd}
\usepackage{mathtools}
\usepackage{mathrsfs}
\usepackage{xspace}
\usepackage[inline]{enumitem}
\usepackage{dsfont}
\usepackage{graphicx}
\usepackage{ctable}

\usepackage{multirow}
\newcommand{\minitab}[2][l]{\begin{tabular}{#1}#2\end{tabular}}

\usepackage{authblk}
\DeclareMathOperator{\E}{\mathsf{E}}
\DeclareMathOperator{\1}{\mathds{1}}
\DeclareMathOperator{\argmin}{\arg\!\min}

\newtheorem{lemma}{Lemma}
\newtheorem{proposition}{Proposition}

\title{A recursion-free functional approximation for the dynamic inventory problem}

\author{Onur A. Kilic}
\affil{\small Department of Operations, University of Groningen, Groningen, The Netherlands}

\author{S. Armagan Tarim}
\affil{\small Department of Business Information Systems, University College Cork, Cork, Ireland}

\date{
}

\begin{document}

\maketitle

\begin{abstract}
\noindent We consider the dynamic inventory problem with non-stationary demands. It has long been known that non-stationary~$(s,S)$ policies are optimal for this problem. However, finding optimal policy parameters remains a computational challenge as it requires solving a large-scale stochastic dynamic program. To address this, we devise a recursion-free approximation for the optimal cost function of the problem. This enables us to compute policy parameters heuristically, without resorting to a stochastic dynamic program. The heuristic is easy-to-understand and -use since it follows by elementary methods of convex minimization and shortest paths, yet it is very effective and outperforms earlier heuristics. 

\end{abstract}

\newpage

\section{Introduction}

Today, industries are experiencing non-stationary demand more frequently as product life cycles are getting increasingly shorter in response to fast technological progress and rapid changes in consumer preferences \citep{Simchi-Levi2003,Chopra2007}. When a product life cycle spans a short period of time, the magnitude of non-stationarity becomes drastic because demand rate changes rapidly as a product moves from one phase of the life cycle to another. Also, in most environments, demand is often heavily seasonal and has a significant trend. Therefore, the demand rate also changes within the phases of the product life cycle. The applicability of stationary inventory control methods in such environments is very limited. Hence, firms must employ inventory policies which can effectively match their supply to non-stationary demand \citep{Kurawarwala1996,Graves2008}. 

Managing inventories is more challenging when demand is non-stationary. This is because fluctuations in demand must be reflected in replenishments. In other words, non-stationary demand compels non-stationary inventory control. The complexity of managing inventories is further intensified when replenishments require fixed costs. Here, the non-stationarity of demand affects not only the size but also the timing of replenishments. This leads to a dynamic inventory problem where replenishment decisions must be made considering possible future demands besides impending inventory needs.

The characterization of optimal control policies for inventory problems has always been a center of interest in the inventory management literature. In this context, \citeauthor{Scarf1960}'s \citeyearpar{Scarf1960} proof of the optimality of~$(s,S)$ policies for the finite-horizon dynamic inventory problem is of particular importance. This seminal result paved the way for a large number of studies. \cite{Iglehart1963} showed that the optimal policy converges to a stationary~$(s,S)$ policy if the planning horizon is sufficiently long. This is followed by a variety of studies on exact and approximate methods to compute the optimal stationary~$(s,S)$ policy \citep[see e.g.][]{Veinott1965,Wagner1965,Johnson1968,Sivazlian1971,Archibald1978, Ehrhardt1979,Porteus1979,Federgruen1984,Zheng1991,Feng2000}. 

\citeauthor{Scarf1960}'s \citeyearpar{Scarf1960} proof the optimality of~$(s,S)$ policies immediately extends to non-stationary demands. However, the computation of non-stationary $(s,S)$ policies requires solving a large-scale stochastic dynamic program where a series of cost functions should be computed for each and every period in the planning horizon recursively. As mentioned by many scholars \citep[see e.g.][]{Karlin1960,Silver1978,Bollapragada1999,Neale2009}, such a numerical procedure is very complex and computationally expensive to be implemented in practice. Notwithstanding, not much work has been done on the non-stationary problem as compared to its stationary counterpart. \cite{Silver1978} and \cite{Askin1981} developed heuristics which determine policy parameters by using the least expected cost per period criterion. \cite{Bollapragada1999} proposed a conceptually different approach where the non-stationary problem is approximated by a series of stationary problems. These stationary problems are constructed by averaging demands over a number of consecutive periods based on an expected cycle time, and the policy parameters of each stationary problem are computed by means of stationary analysis.  \citet{Xiang2018} recently presented a method where policy parameters are approximated by iteratively solving a series of mixed integer non-linear programs for each period of the planning horizon. While effective, this method is computationally too expensive for practical use. 

We follow the aforementioned line of research and develop a new approach to compute policy parameters heuristically. Our approach relies on the idea of approximating the non-convex cost function by a sequence of convex cost functions associated with prospective replenishment cycles. The approximate cost function is recursion-free. It enables us to establish policy parameters by solving a series of root-finding problems and a single deterministic shortest path problem. We numerically compare the new heuristic against earlier heuristics by \cite{Askin1981} and \cite{Bollapragada1999} as well as the optimal stochastic dynamic program. The results show that our heuristic very effective and outperforms earlier heuristics. 

The rest of the paper is organized as follows. In Section~\ref{sec:definition}, we present a formal definition of the problem and provide the details of the optimal policy. In Section~\ref{sec:approx} we introduce our approximation of the optimal cost function. In Section~\ref{sec:heur}, we present our heuristic approach. In Section~\ref{sec:example}, we provide an illustrative example. In Section~\ref{sec:computation}, we present computational refinements that improve the computational efficiency of the heuristic. In Section~\ref{sec:numerical}, we conduct a numerical study for performance assessment. In Section~\ref{sec:stationary}, we discuss a stationary adaptation of the proposed approach and draw analogies with the existing literature. In Section~\ref{sec:conc}, we conclude.

\section{Problem definition and preliminaries}
\label{sec:definition}

We consider a finite-horizon periodic-review inventory system facing stochastic and non-stationary demands. The planning horizon consists of~$T$ time periods. Demands over these periods are denoted as $\xi_1,\ldots,\xi_T$. These are non-negative random variables which are independent but not necessarily identically distributed. We assume that replenishment orders are placed at the beginning of each period and received instantaneously. Then, demand is realized. Excess demand is backordered. There is a holding cost~$h$ per unit of on-hand inventory and a penalty cost~$p$ per unit of backlog carried from one period to the next. Each replenishment order incurs a fixed cost~$K$. The objective is to find a policy that minimizes the expected total cost over the planning horizon. We keep the problem definition restrictive for the sake of notational brevity. Our analysis and methods can accommodate systems with procurement and salvage costs, lead time, and discounting.

Throughout the paper, we let~$\1\{\cdot\}$ be the indicator function and~$\E$ be the expectation operator. We also define $x^+=\max\{x,0\}$ and $x^-=\max\{0,-x\}$. 

In what follows, we summarize some earlier results on the problem which will be functional in our analysis. We refer the reader to \citet{Scarf1960} for a complete analysis of the structure of the optimal inventory policy. 

We first consider the single-period cost function which provides the expected holding and penalty costs to be charged in period~$n$ if the after-replenishment inventory level is~$y$. This can be written as
\begin{align}\label{eq:Ln}
L_n(y) = h\E (y-\xi_n)^+ + p\E (y-\xi_n)^-
\end{align}
which is convex as holding and penalty costs are linear.

The optimal cost function providing the expected total costs from period~$n$ onwards satisfies the stochastic dynamic program 
\begin{align} \label{eq:Cn}
C_n(x) = \min_{x\leq y} \{K \1\{x<y\} + L_n(y) + \E C_{n+1}(y-\xi_n)\} 
\end{align}
where~$x$ and~$y$ respectively stand for the before- and after-replenishment inventory levels. The terminal cost reads~$C_{T+1}(x)=0$ for all $x$.

We can also write the optimal cost function with respect to the after-replenishment inventory level~$y$ as
\begin{align}\label{eq:Gn}
G_n(y) = L_n(y) + \E C_{n+1}(y-\xi_n).
\end{align}

\cite{Scarf1960} showed that~$C_n$ and~$G_n$ satisfy a functional property which he referred to as $K$-convexity. This property is sufficient to prove that the optimal replenishment policy in any period~$n$ is an~$(s_n,S_n)$ policy. That is, a replenishment order is issued in period~$n$ if the initial inventory level is below a re-order level~$s_n$, and, if so, the order quantity should increase the inventory level to an order-up-to level~$S_n$. 

The optimal policy parameters immediately follow from~$G_n$. The order-up-to level is the minimizer
\begin{align}\label{eq:Sn}
S_n = \argmin_y G_n(y).
\end{align}

The re-order level is the smallest inventory level whose cost exceeds the cost of the order-up-to level by a margin no more than the fixed replenishment cost. That is
\begin{align}\label{eq:sn}
s_n = \min\{y\mid G_n(y) \leq G_n(S_n) + K\}.
\end{align}

The aforementioned results fully characterize the parameters of the optimal policy. Nevertheless, finding these parameters is demanding as it requires computing a continuous cost function for each and every period recursively. It is possible to detour the continuity issue by assuming discrete demands \citep[see e.g.][]{Bollapragada1999,Lulli2004,Guan2008}. However, the difficulty prevails because one needs to consider all possible inventory levels and demand realizations in each period, both of which could be arbitrarily large in number.

%

\section{The approximate cost function}
\label{sec:approx}

We now introduce an approximate cost function to avoid the extensive computational effort required to construct the optimal cost function recursively for each and every period. 

Let us begin by introducing the multi-period version of the single-period cost function in \eqref{eq:Ln}. Suppose a procurement lot becomes available in period~$n$ and increases the inventory level to~$y$. Also assume that the next procurement lot will not be available until period~$n+a$ ($a\geq 1$). That is, the length of the imminent replenishment cycle is~$a$ periods. We can write the expected total inventory costs over this replenishment cycle as
\begin{align}\label{eq:Lna}
L_{na}(y) = \sum_{k=1}^a \left(h\E(y-\xi_{nk})^+ + p\E(y-\xi_{nk})^- \right)
\end{align}
where $\xi_{nk}=\xi_n+\ldots+\xi_{n+k-1}$.

The multi-period cost function~$L_{na}$ is convex as it is a sum of convex functions. Its derivative reads $L_{na}^{'}(y)=-ap+(h+p) \sum_{k=1}^a \varphi_{nk}(y)$ where $\varphi_{nk}$ is the cumulative distribution function of accumulated demand~$\xi_{nk}$. Therefore, its minimizer~$y_{na}$ can be obtained by means of the optimality condition
\begin{align}\label{eq:yna}
y_{na} = \min\left\{y \,\middle\vert\, \frac{1}{a} \sum_{k=1}^a \varphi_{nk}(y) \geq \frac{p}{h+p}\right\}
\end{align}
and computed efficiently with a root-finding algorithm---provided that the distribution functions of accumulated demands are available.

We now consider an alternative inventory problem which we use as a proxy to introduce our approximation. This is the same as the original problem, except at each replenishment period the decision maker freely chooses a post-replenishment inventory level (i.e. negative order quantities are allowed) and explicitly specifies the timing of the next replenishment period. The essence of this problem is that the inventory costs in successive replenishment cycles are independent of each other and can be optimized myopically. Hence, it can be modeled as the following deterministic dynamic program 
\begin{align}\label{eq:vn}
v_n = \min_{1\leq a\leq T-n+1}\{\ell_{na} + v_{n+a}\}
\end{align}
where $\ell_{na}=K+L_{na}(y_{na})$ and $v_{T+1}=0$. It is easy to see that~\eqref{eq:vn} is a single-source shortest path problem. Hence, we can efficiently compute the optimal costs~$v_n$ and cycle lengths~$a_n$ for all periods in the planning horizon. 


We are now ready to present our approximation~$\hat{G}_n$ of the optimal cost function~$G_n$. The approximate cost function reads
\begin{align} \label{eq:Ghatn}
\hat{G}_n(y) = \min_{1\leq a\leq T-n+1} \{L_{na}(y) + v_{n+a}\}.
\end{align}

The approximate cost function has two attractive properties. First, it is recursion-free. It builds solely on the multi-period cost function~\eqref{eq:Lna} and the solution of the dynamic program~\eqref{eq:vn}. Second, its minimum can be attained very efficiently. Because it is defined as the minimum of a sequence of convex functions, it has multiple minima like the optimal cost function. But these local minima are aligned with the minima of its convex components, all of which can be obtained from~\eqref{eq:yna}.

\section{Heuristic approach}
\label{sec:heur}

Having established the approximate cost function, we proceed with our heuristic approach to compute approximate policy parameters. To that end, our course of action is simple. We adopt the definitions in~\eqref{eq:Sn} and~\eqref{eq:sn} but use the approximate cost function~$\hat{G}_n$ instead of the optimal cost function~$G_n$. 

Following~\eqref{eq:Sn}, we set the approximate order-up-to level~$\hat{S}_n$ as the minimizer of~$\hat{G}_n$. Hence, we have
\begin{align}\label{eq:Shatn}
\hat{S}_n = y_{na_n}
\end{align}
which becomes available once~\eqref{eq:vn} has been solved.

In line with~\eqref{eq:sn}, we set the approximate re-order level~$\hat{s}_n$ as the smallest inventory level whose cost exceeds the minimum of the approximate cost function by a margin no more than the fixed replenishment cost~$K$. It follows from~\eqref{eq:vn} and~\eqref{eq:Ghatn} that the minimum of the approximate cost function is~$v_n-K$. Thus, we have
\begin{align}\label{eq:shatn}
\begin{split}
\hat{s}_n	= \min\{y\mid \hat{G}_n(y) \leq v_n\}. 
\end{split}
\end{align}

Because $\hat{G}_n(y)$ is defined as the minimum of convex functions $L_{na}(y)+v_{n+a}$ for different values of~$a$, we can re-write~\eqref{eq:shatn} as
\begin{align}\label{eq:shatn_cycle}
\hat{s}_n = \min_{1\leq a\leq T-n+1}\min\{y\mid L_{na}(y) + v_{n+a} \leq v_n\}
\end{align}
which suggests that~$\hat{s}_n$ can be obtained by applying a root-finding routine on each convex component independently.

\section{An illustrative example}
\label{sec:example}

We now provide a simple numerical example to better illustrate the approximate cost function and the heuristic approach. Let us consider a 4-period problem instance with cost parameters $h=1$, $p=10$, and $K=100$. The period demands are discrete random variables uniformly distributed around their average values as $\xi_n\sim \mathcal{U}(\mu_n-\Delta,\mu_n+\Delta)$ where $\mu=[60,15,30,40]$ and $\Delta=10$. 

\begin{figure}[htbp]
\vspace{.2cm}
\centering
\includegraphics[width=\textwidth]{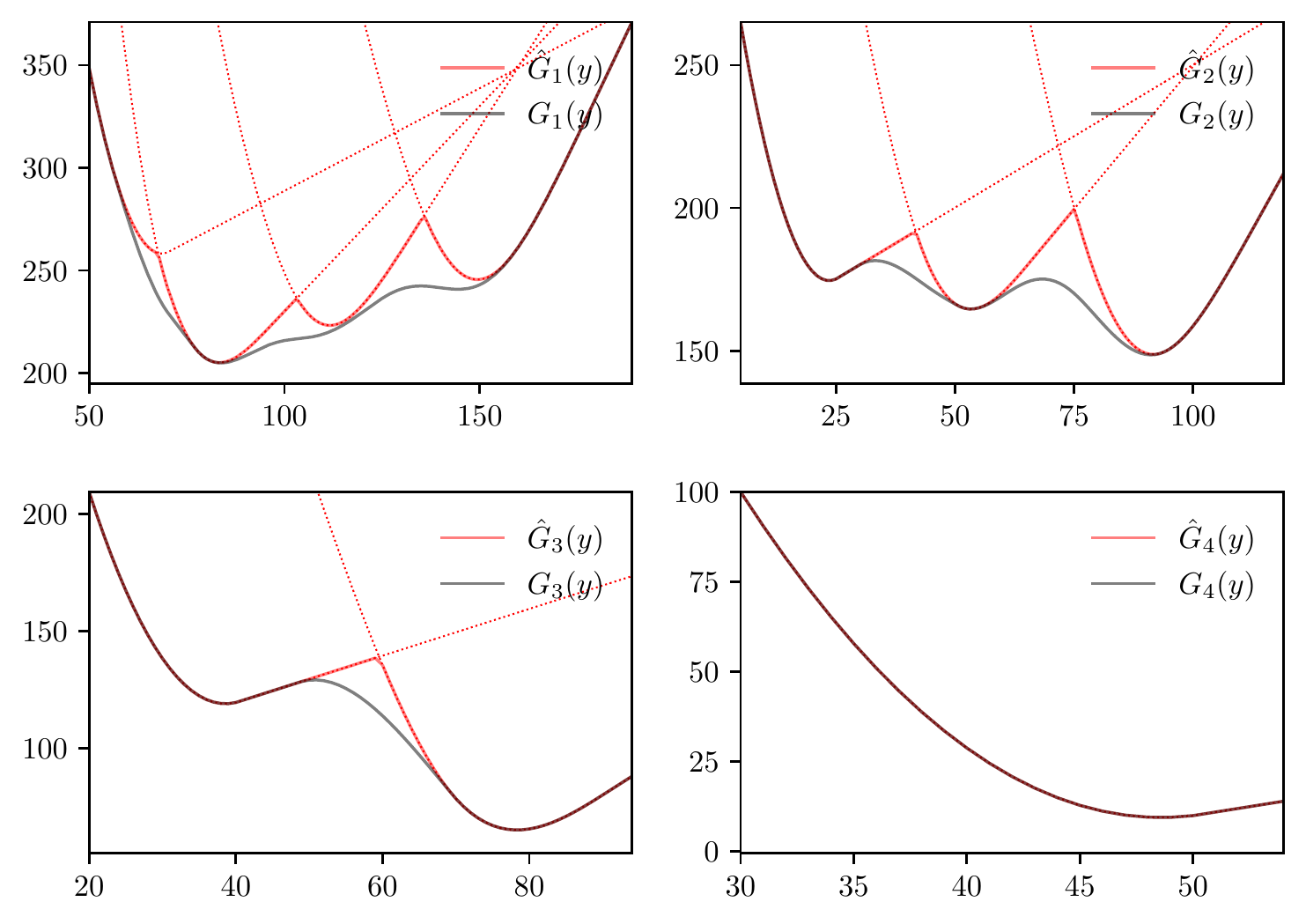}
\caption{The optimal and approximate cost functions.}
\label{fig:functions}
\end{figure}

Figure~\ref{fig:functions} illustrates the optimal cost function~$G_n$ and the approximate cost function~$\hat{G}_n$ (as well as its convex components) for each period of the problem instance. The plots are insightful. We observe that while~$\hat{G}_n$ is not a precise approximation of~$G_n$ over the entire domain, the oscillations of~$G_n$ and~$\hat{G}_n$ are strongly coupled with each other. This suggests that the local minimizers of~$\hat{G}_n$ are good approximations for the local minimizers of~$G_n$, which is critical for the performance of the heuristic approach. We also see that~$G_n$ and~$\hat{G}_n$ are the same for the last period as both are equivalent to the single-period cost function.

\begin{table}[htbp]
\vspace{0.2cm}
\centering
\begin{tabular*}{24pc}{@{\extracolsep{\fill}}crrrrrr}
\toprule
	& \multicolumn{3}{c}{Optimal} & \multicolumn{3}{c}{Heuristic}	\\
\cmidrule{2-4} \cmidrule{5-7} 
$n$	&	$s_n$ &	$S_n$ &	$G_n(S_n)$	&	$\hat{s}_n$ &	$\hat{S}_n$	& $\hat{G}_n(\hat{S}_n)$	\\
\midrule
 1 &     56 &     84 &  204.97 &     56 &     83 &  205.16 \\
 2 &      7 &     91 &  148.55 &      7 &     92 &  148.74 \\
 3 &     26 &     78 &   65.08 &     26 &     78 &   65.08 \\
 4 &     30 &     49 &    9.52 &     30 &     49 &    9.52 \\
\bottomrule
\end{tabular*}
\caption{The optimal and approximate policy parameters and expected costs.}
\label{tab:parameters}
\end{table}

Table~\ref{tab:parameters} reports the optimal and approximate policy parameters as well as the minima of the optimal and approximate cost functions. We observe that the parameters of the optimal and approximate policies are very close. If there is no initial inventory, both policies place a replenishment order in the first period. Hence, the associated cost functions suggest that the expected total costs of the optimal and approximate policies are~$K+G_1(S_1)=304.97$ and~$K+\hat{G}_1(\hat{S}_1)=305.16$. The cost figure is exact for the optimal policy. But it is not necessarily accurate for the approximate policy. Indeed, an exact numerical analysis reveals that the expected cost of the approximate policy is~$305.04$. This indicates an optimality gap smaller than~0.025\%.

\section{Computational methods}
\label{sec:computation}

The analysis carried out thus far completely characterizes the heuristic. We now discuss the implementation details of the heuristic and establish structural properties that significantly improve its computational performance. 

\subsection{Procedure}

The implementation of the heuristic proceeds with the following steps. 
\begin{enumerate}
\item The optimal inventory levels and corresponding costs are computed based on the associated multi-period cost function, see~\eqref{eq:Lna} and~\eqref{eq:yna}. This can be done by using a standard root-finding algorithm. The procedure requires the distribution functions of accumulated demands. These may be readily available in some cases. For instance, if demand follows a normal distribution in each period then the accumulated demand over any interval also follows a normal distribution. Otherwise, they should be obtained by numerical convolution. 
\item The dynamic program is solved, see~\eqref{eq:vn}. This is a single-source shortest path problem in a directed acyclic graph where periods and replenishment cycles respectively stand for vertices and edges, and the minimum costs of replenishment cycles correspond to the edge costs. The problem can be solved efficiently by means of a standard recursion. The time-complexity of such a procedure is linear in the number of replenishment cycles. The solution of the dynamic program yields a cycle length and an approximate cost figure for each period. 
\item The approximate policy parameters are computed for each period. The order-up-to levels immediately follow from the cycle lengths, see \eqref{eq:Shatn}. The re-order levels can be obtained by applying a standard root-finding procedure on each convex component of the approximate cost function, see \eqref{eq:shatn_cycle}.
\end{enumerate}

\subsection{Performance improvements}

It is clear that each step of the computational procedure scale with the number of replenishment cycles, which grows quadratically in the number of periods. In the following, we show that evaluating only a few replenishment cycles per period is sufficient to implement the heuristic.

The next proposition establishes some basic properties that are functional in the proofs of Lemma~\ref{lem:vn_simple} and Lemma~\ref{lem:shatn_cycle_simple}.

\begin{proposition} \label{prop:basics}
For any period~$n$ and consecutive cycle lengths~$a$ and~$a+1$,
\begin{enumerate}[label=(\alph*)]
\item $\varphi_{na}(y) \geq \varphi_{n,a+1}(y)$ for all~$y$,
\item $\frac{1}{a} \sum_{k=1}^a \varphi_{nk}(y) \geq \frac{1}{a+1} \sum_{k=1}^{a+1} \varphi_{nk}(y)$ for all~$y$,
\item $y_{na}\leq y_{n,a+1}$, and
\item $L_{na}'(y)\geq L_{n,a+1}'(y)$ for all $y\leq y_{na}$.
\end{enumerate}
\end{proposition}

\begin{proof}
\begin{enumerate}[label=(\alph*)]
\item Follows by definition as~$\varphi_{na}$ and~$\varphi_{n,a+1}$ are cumulative distribution functions of accumulated non-negative demands. 
\item Immediately follows from~(a).
\item Suppose $y_{na}>y_{n,a+1}$. Then, from~\eqref{eq:yna} we must have on the one hand $\frac{1}{a} \sum_{k=1}^a \varphi_{nk}(y_{n,a+1})<\frac{p}{h+p}$ and on the other hand $\frac{p}{h+p}\leq \frac{1}{a+1} \sum_{k=1}^{a+1} \varphi_{nk}(y_{n,a+1})$. This cannot hold due to~(b). 
\item It follows from~\eqref{eq:Lna} that $L'_{na}(y)-L'_{n,a+1}(y)=p-(h+p)\varphi_{n,a+1}(y)$. Because $y\leq y_{na}$, we have from~\eqref{eq:yna} that $\frac{p}{h+p}\geq \frac{1}{a} \sum_{k=1}^{a} \varphi_{nk}(y)$. Due to~(a), this implies $\frac{p}{h+p} \geq \varphi_{n,a+1}(y)$. Therefore, we have $p-(h+p)\varphi_{n,a+1}(y)\geq 0$.
\end{enumerate}
\end{proof}

\begin{lemma} \label{lem:vn_simple}
The dynamic program~\eqref{eq:vn} can be simplified as
\begin{align}\label{eq:vn_simple}
v_n = \min_{1\leq a\leq \bar{a}_n}\left\{\ell_{na} + v_{n+a}\right\}
\end{align}
where $\bar{a}_n = \max\{a\mid L_{n1}(y_{na}) \leq \ell_{n1}\}$.
\end{lemma}

\begin{proof}
It is clear that a replenishment cycle that starts in period~$n$ and spans the next~$a$ periods cannot be a part of an optimal replenishment plan if $\ell_{na}$ exceeds the sum of $\ell_{n1}$ and $\ell_{n+1,a-1}$, as it can then be replaced with two cycles such that the first starts and ends in period~$n$ and the second starts in~period~$n+1$ and ends in period~$n+a$. This condition can be formalized as
\begin{align*}
\ell_{na} & > \ell_{n1} + \ell_{n+1,a-1} \\
K + L_{n1}(y_{na}) + \E L_{n+1,a-1}(y_{na}-\xi_{n1}) - \ell_{n+1,a-1} & > \ell_{n1} \\
K + \E L_{n+1,a-1}(y_{na}-\xi_{n1}) - \ell_{n+1,a-1} & > \ell_{n1} - L_{n1}(y_{na})
\end{align*}
where $\ell_{na}=K+L_{na}(y_{na})=K+L_{n1}(y_{na})+\E L_{n+1,a-1}(y_{na}-\xi_{n1})$ follows from~\eqref{eq:Lna}. Because~$\ell_{n+1,a-1}$ is the minimum of $K+L_{n+1,a-1}(y)$, we have $K+L_{n+1,a-1}(y_{na}-\xi_{n1})\geq \ell_{n+1,a-1}$ for any realization of~$\xi_{n1}$. Hence, the left-hand-side of the last inequality is non-negative. This suggests that the condition holds if the right-hand-side is negative, which is indeed the case if $L_{n1}(y_{na})>\ell_{n1}$. The proof is completed if we show that $L_{n1}(y_{na})\leq L_{n1}(y_{n,a+1})$. This follows from Proposition~\ref{prop:basics}(c) and the convexity of~$L_{n1}$. 
\end{proof}

\begin{lemma}\label{lem:shatn_cycle_simple}
The expression of the approximate re-order level in~\eqref{eq:shatn_cycle} can be simplified as
\begin{align}\label{eq:shatn_cycle_simple}
\hat{s}_n = \min_{1\leq a\leq a_n}\min\{y\mid L_{na}(y) + v_{n+a} \leq v_n\}.
\end{align}
\end{lemma}

\begin{proof}
For the sake of brevity, let us define $f_{na}(y)=L_{na}(y)+v_{n+a}$ and $\hat{s}_{na}=\min\{y\mid f_{na}(y)\leq v_n\}$. Then we can re-write~\eqref{eq:shatn_cycle} as $\hat{s}_n=\min_{1\leq a\leq T-n+1} \hat{s}_{na}$. Note that~$\hat{s}_{na}$ may not exist for some~$a$, as the minimum~$f_{na}(y_{na})$ could be larger than~$v_n$. We omit these in the following. The proof follows if $\hat{s}_{na_n}\leq \hat{s}_{na}$ holds for any~$a$ such that $a\geq a_n$. Suppose~$a\geq a_n$. Because $f_{na_n}(y)$ and $f_{na}(y)$ are convex and $y_{na_n}\leq y_{na}$ due to Proposition~\ref{prop:basics}(c), both functions are decreasing in the interval $y\leq y_{na_n}$. From Proposition~\ref{prop:basics}(d), we have that $f_{na_n}'(y)\geq f_{na}'(y)$ for all $y\leq y_{na}$. This suggests that $f_{na_n}(y)$ decreases less than $f_{na}(y)$ over the interval $\hat{s}_{na}\leq y\leq y_{na_n}$. Thus we have $f_{na_n}(\hat{s}_{na})-f_{na_n}(y_{na_n})\leq f_{na}(\hat{s}_{na})-f_{na}(y_{na_n})$. It follows from the definition of~$a_n$ that $f_{na_n}(y_{na_n})\leq f_{na}(y_{na_n})$. Hence, we have $f_{na_n}(\hat{s}_{na})\leq f_{na}(\hat{s}_{na})$. This implies $\hat{s}_{na_n}\leq \hat{s}_{na}$. 
\end{proof}

The results presented in Lemma~\ref{lem:vn_simple} and Lemma~\ref{lem:shatn_cycle_simple} have important ramifications on the computational performance of the heuristic. Lemma~\ref{lem:vn_simple} defines a set of replenishment cycles for each period by setting an upper bound on the cycle length, and shows that an optimal replenishment plan is a collection of these replenishment cycles. This drastically reduces the number of multi-period cost functions to be evaluated and (if necessary) the distribution functions of accumulated demands to be computed. Besides, it leads to a direct reduction in the computational time associated with the dynamic program. In a similar vein, Lemma~\ref{lem:shatn_cycle_simple} shows that the approximate re-order level can be obtained by evaluating only a limited number of convex components of the approximate cost function.

\section{Numerical Study} 
\label{sec:numerical}

In this section, we conduct a numerical study and compare the new heuristic against the earlier heuristics by \cite{Askin1981} and \cite{Bollapragada1999}. In the following, we first present the numerical design and then discuss our findings.

\subsection{Design} 

In our numerical study, conduct a full factorial analysis involving a large variety of problem instances which differ with respect to their cost parameters and demand specifications. The following values are used for cost parameters; holding cost $h=1$, penalty costs $p=5,10,20$, and fixed replenishment costs $K=800,3200,12800$. The fixed replenishment costs are chosen such that they correspond to cycle lengths of~$4$,~$8$ and~$16$ periods in the economic order quantity problem with a demand of~$100$ units per period. We reflect on time-varying demands by means of different demand patterns which define how average demand progresses over the planning horizon. We adopt demand patterns from the empirical data reported by \cite{Kurawarwala1996}. These involve monthly demands for four PC products over their life-cycles and demonstrate the combined effect of product life-cycles and seasonal fluctuations. We slightly adapt the original data for convenience. We first transform the data into a weekly scale by means of polynomial interpolation, while keeping the length of the product life-cycle intact. This leads to planning horizons varying around 70--120 periods. Then, we scale demands such that the average demand over the entire planning horizon is~$100$ units. Figure~\ref{fig:patterns} graphically illustrates all four demand patterns---referred to as $\mu=\text{Emp1, Emp2, Emp3, Emp4}$. We parametrize demand uncertainty by a fixed coefficient of variation over the planning horizon. This enables us to investigate the performance of heuristics with respect to varying levels of demand uncertainty in a unified fashion. We use normal and negative binomial distributions to model demands as these distributions can be characterized by their mean and coefficient of variation independently. We shall see that demand uncertainty has a major impact on the heuristic performance. To better investigate this impact, we consider two classes of instances. The first reflects on moderate levels of demand uncertainty. It involves normally distributed demands with coefficients of variation $\rho=0.10,0.20,0.30$. We discretize and truncate normal distributions with a unit stepsize between zero and twice their mean, as the optimal policy can only be computed exactly for discrete demands. The second demonstrate high levels of demand uncertainty. It involves negative binomially distributed demands with coefficients of variation $\rho=0.50,0.75,1.00$.

\begin{figure}[tbp]
\centering
\includegraphics[width=\textwidth]{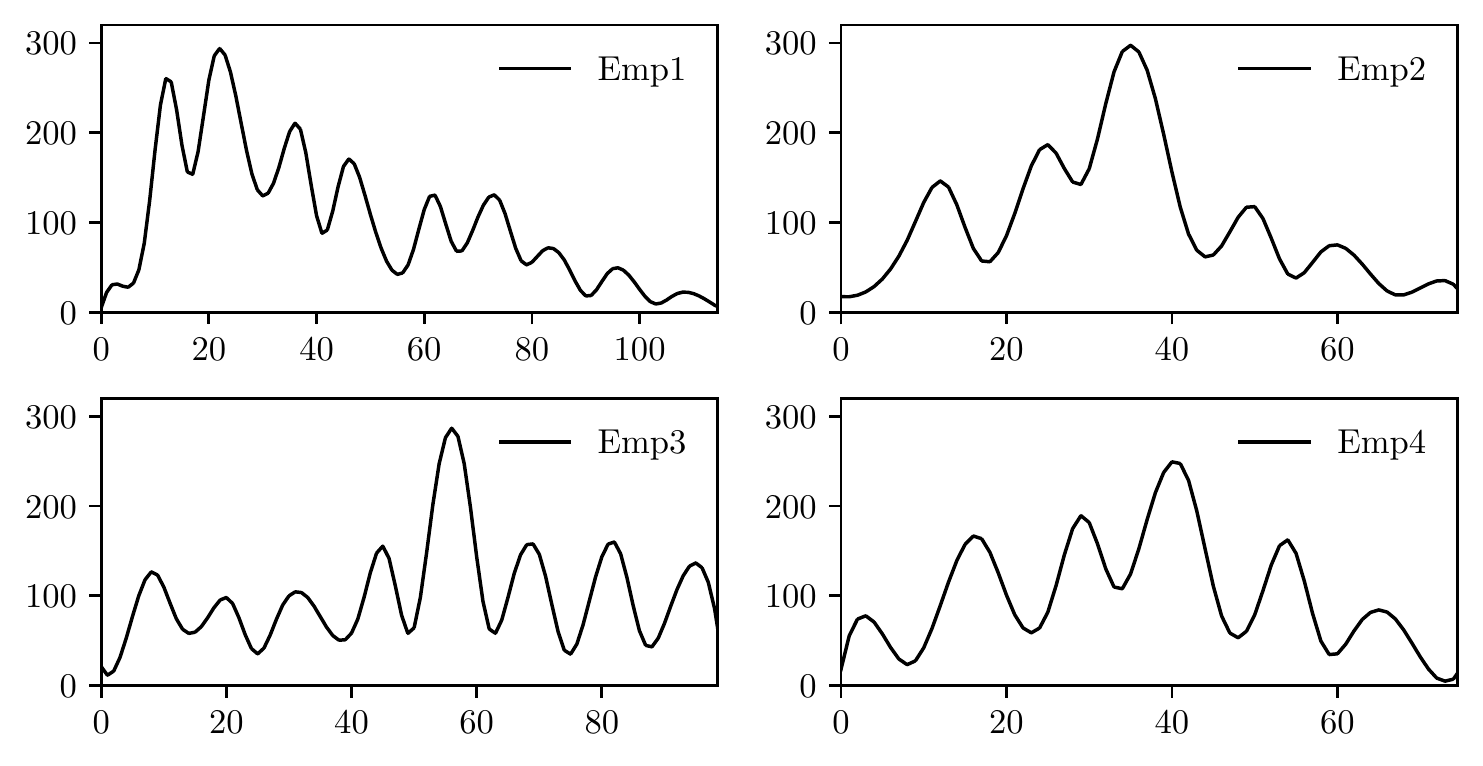}
\caption{Demand patterns. \label{fig:patterns}}
\end{figure}

Our numerical design leads to a test bed of~$216$ problem instances. For each instance, we compare the expected cost of employing the optimal policy against those of the new heuristic as well the earlier heuristics by \citet{Askin1981} and \citet{Bollapragada1999}. We refer the reader to the associated papers for the details of the benchmark heuristics. \citeauthor{Bollapragada1999}'s \citeyearpar{Bollapragada1999} heuristic cannot account for end-of-horizon effects as it implicitly assumes an infinite planning horizon. To alleviate this drawback and ensure a fair comparison, we replace the parameters of all heuristics with optimal policy parameters for the last~$18$ periods of the planning horizon---as in \citet{Bollapragada1999}. We do not present a detailed discussion on the computational times as all heuristics scale very well. The run times for each of the three heuristics are in the order of seconds for the test problems considered.


\subsection{Results}

\begin{table}[tbp]
\centering
\begin{tabular*}{\textwidth}{@{\extracolsep{\fill}}lrrrrrrr}
\toprule
          &        & \multicolumn{2}{c}{AH} & \multicolumn{2}{c}{BMH} & \multicolumn{2}{c}{NH} \\
\cmidrule{3-4} \cmidrule{5-6} \cmidrule{7-8}
Parameter &  Value &     Avg. & Max. &      Avg. &  Max.  &     Avg. &  Max. \\
\midrule
\multirow{3}{*}{\minitab[l]{$K$}}
     &    800 &         2.52 &  3.90 &          3.42 &   5.46 &         0.33 &  0.75 \\
     &   3200 &         2.41 &  4.35 &          5.60 &   8.05 &         0.23 &  0.79 \\
     &  12800 &         1.79 &  4.16 &          6.88 &  13.19 &         0.08 &  0.37 \\
\midrule
\multirow{3}{*}{\minitab[l]{$p$}}
     &      5 &         2.65 &  4.16 &          4.45 &   8.05 &         0.13 &  0.50 \\
     &     10 &         1.99 &  3.40 &          5.29 &  11.82 &         0.20 &  0.75 \\
     &     20 &         2.08 &  4.35 &          6.16 &  13.19 &         0.31 &  0.79 \\
\midrule
\multirow{4}{*}{\minitab[l]{$\mu$}}
     &   Emp1 &         2.47 &  3.81 &          4.73 &   8.27 &         0.19 &  0.63 \\
     &   Emp2 &         2.07 &  4.35 &          4.78 &  10.42 &         0.20 &  0.73 \\
     &   Emp3 &         2.50 &  4.16 &          5.28 &   8.05 &         0.24 &  0.73 \\
     &   Emp4 &         1.92 &  3.90 &          6.40 &  13.19 &         0.23 &  0.79 \\
\midrule
\multirow{3}{*}{\minitab[l]{$\rho$}}
     &   0.10 &         2.05 &  4.16 &          5.87 &  13.19 &         0.05 &  0.22 \\
     &   0.20 &         2.25 &  3.85 &          5.36 &  11.73 &         0.21 &  0.62 \\
     &   0.30 &         2.42 &  4.35 &          4.66 &  10.27 &         0.38 &  0.79 \\
\midrule
\multicolumn{2}{l}{\multirow{1}{*}{All instances}} 
              &         2.24 &  4.35 &          5.30 &  13.19 &         0.21 &  0.79 \\
\bottomrule
\end{tabular*}
\caption{Results on instances with moderate demand uncertainty.\label{tab:res_moderate}}
\end{table}

\begin{table}[h!]
\centering
\begin{tabular*}{\textwidth}{@{\extracolsep{\fill}}lrrrrrrr}
\toprule
          &        & \multicolumn{2}{c}{AH} & \multicolumn{2}{c}{BMH} & \multicolumn{2}{c}{NH} \\
\cmidrule{3-4} \cmidrule{5-6} \cmidrule{7-8}
Parameter &  Value &     Avg. & Max. &      Avg. &  Max.  &     Avg. &  Max. \\
\midrule
\multirow{3}{*}{\minitab[l]{$K$}}
     &    800 &         4.68 &   7.94 &          2.51 &  3.92 &         1.07 &  1.71 \\
     &   3200 &         4.34 &   9.55 &          3.10 &  5.12 &         1.52 &  2.27 \\
     &  12800 &         4.60 &  15.73 &          2.96 &  7.29 &         1.16 &  2.64 \\
\midrule
\multirow{3}{*}{\minitab[l]{$p$}}
     &      5 &         4.12 &  15.73 &          1.98 &  3.94 &         0.95 &  2.05 \\
     &     10 &         4.34 &   9.35 &          2.68 &  5.45 &         1.24 &  2.25 \\
     &     20 &         5.16 &  13.98 &          3.90 &  7.29 &         1.56 &  2.64 \\
\midrule
\multirow{4}{*}{\minitab[l]{$\mu$}}
     &   Emp1 &         4.61 &   8.85 &          2.70 &  5.10 &         1.20 &  2.18 \\
     &   Emp2 &         6.02 &  15.73 &          2.64 &  5.78 &         1.20 &  2.64 \\
     &   Emp3 &         3.54 &   6.37 &          2.94 &  5.12 &         1.30 &  2.27 \\
     &   Emp4 &         3.99 &   7.94 &          3.14 &  7.29 &         1.30 &  2.45 \\
\midrule
\multirow{3}{*}{\minitab[l]{$\rho$}}
     &   0.50 &         3.20 &   5.76 &          3.54 &  7.29 &         0.85 &  2.27 \\
     &   0.75 &         4.63 &  12.66 &          2.73 &  4.58 &         1.27 &  2.64 \\
     &   1.00 &         5.79 &  15.73 &          2.30 &  4.12 &         1.63 &  2.56 \\
\midrule
\multicolumn{2}{l}{\multirow{1}{*}{All instances}} 
              &         4.54 &  15.73 &          2.86 &  7.29 &         1.25 &  2.64 \\
\bottomrule
\end{tabular*}
\caption{Results on instances with high demand uncertainty.\label{tab:res_high}}
\end{table}

We provide a summary of our results on instances with moderate and high demand uncertainty in Table~\ref{tab:res_moderate} and Table~\ref{tab:res_high}. These report average and maximum (percentage) optimality gaps of each heuristic for all problem instances characterized by the same pivot parameter. We abbreviate the new heuristic, \citeauthor{Askin1981}'s heuristic and \citeauthor{Bollapragada1999}'s heuristic as NH, AH, and BMH, respectively. 

The results clearly show that NH is very effective. It averages an optimality gap of~$0.21\%$ for instances with moderate demand uncertainty and~$1.25\%$ for instances with high demand uncertainty. NH significantly outperforms the benchmark heuristics. Its maximum optimality gaps are~$0.79\%$ and~$2.64\%$ for instances with moderate and high demand uncertainty. These figures are even smaller than the average optimality gaps of the benchmark heuristics---$2.24\%$ and~$4.54\%$ for AH and~$5.30\%$ and~$2.86\%$ for BMH. The performance of NH is consistent across different problem settings. It is yet affected by cost parameters and demand uncertainty. The fixed replenishment has a minor impact. NH's performance change only slightly when the fixed replenishment cost increases more than tenfold. The impact of penalty cost is more visible. NH's optimality gap tends to increase with penalty costs. These findings mostly hold for AH and BMH as well as NH. We observe that demand uncertainty is a major determinant of heuristic performance. This becomes particularly apparent when we compare the results on instances with moderate and high demand uncertainty. NH's performance deteriorates with the extent of uncertainty. Nevertheless, its optimality gap is much lower than AH and BMH even for instances with high demand uncertainty.

\section{A stationary analysis of the heuristic} 
\label{sec:stationary}

To provide a better understanding of the proposed heuristic and draw analogies with some well-established results in the literature, we now discuss a stationary adaptation of our heuristic. We consider the infinite-horizon inventory problem with stationary demands focusing on average cost optimality. The optimal control policy for this problem is known to be a stationary~$(s,S)$ policy \citep[see e.g.][]{Iglehart1963a,Zheng1991}.

The adaptation proceeds as follows. The multi-period cost function is now period-independent and can be expressed as $L_a(y)=\sum_{k=1}^a (h\E(y-\xi_{(k)})^+ + p\E(y-\xi_{(k)})^-)$ where~$\xi_{(k)}$ is the $k$-fold convolution of the demand distribution. Let $y_a$ be the minimizer of~$L_a$. Then, we can write the optimal total cost over a replenishment cycle that spans over~$a$ periods as $\ell_a=K+L_a(y_a)$. Because demand is stationary and the objective is to minimize the long-term average cost, we concentrate on the average cost per period. Hence, the dynamic program reduces to choosing the length of a replenishment cycle as $v=\min_a \left\{\ell_a/a\right\}$ such that~$v$ is the average cost per period. Let~$a_*$ be the optimal cycle length. Then we can specify approximate policy parameters as $\hat{S}=y_{a_*}$ and $\hat{s}=\min\{y\mid L_1(y)\leq v\}$. Let us now make the following observations. First, the procedure we employ to find the optimal cycle length and the associated order-up-to level is the same procedure that is used to compute the optimal periodic-review order-up-to policy \citep[see e.g.][]{Rao2003,Liu2012}. Secondly, the condition we use to compute the re-order level is equivalent to \citeauthor{Zheng1991a}'s \citeyearpar{Zheng1991a} characterization of the optimal re-order level, except we replace the optimal long-run average cost with that of the periodic-review order-up-to policy. These provide an intuitive explanation of the conceptual ideas behind our heuristic. We use a sub-optimal but easily computable policy as a proxy to approximate the optimal cost function. Then, we compute policy parameters heuristically following the characterization of the optimal policy parameters, where we use the approximate cost function instead of the optimal one. 

The stationary version of our heuristic as we described above has already been considered in the literature. \citet{Porteus1979} introduced this heuristic. Then, \citet{Freeland1980} numerically investigated its cost-effectiveness and reported favourable results. Interestingly, \citet{Porteus1979} mentioned that the heuristic can be applied to non-stationary problems. However, such an application has been absent to date.

\section{Conclusion}
\label{sec:conc}

It has been pointed by many authors that most practical demand patterns fall into the category of non-stationary stochastic demands. But it is also widely accepted that mathematical models associated with non-stationary stochastic demands are very complicated in terms computational needs and user understanding \citep[see e.g.][]{Silver1981,Silver2008}. In this study, we presented a hybrid---myopic and far-sighted---heuristic which can be viewed as an attempt to bridge this gap between theory and practice. Our numerical findings show that the proposed heuristic significantly outperforms earlier heuristics and yields almost-optimal results for a variety of demand patterns and cost parameters. An important practical benefit of the heuristic is that it is easy-to-understand and -use as it builds on readily available methods of shortest paths and convex minimization. 

\bibliographystyle{chicago} 
\bibliography{ssheur}

\begin{thebibliography}{}

\bibitem[\protect\citeauthoryear{Archibald and Silver}{Archibald and
  Silver}{1978}]{Archibald1978}
Archibald, B.~C. and E.~A. Silver (1978).
\newblock (${s,S}$) policies under continuous review and discrete compound
  poisson demand.
\newblock {\em Management Science\/}~{\em 24\/}(9), 899--909.

\bibitem[\protect\citeauthoryear{Askin}{Askin}{1981}]{Askin1981}
Askin, R. (1981).
\newblock A procedure for production lot sizing with probabilistic dynamic
  demand.
\newblock {\em IIE Transactions\/}~{\em 13\/}(2), 132--137.

\bibitem[\protect\citeauthoryear{Bollapragada and Morton}{Bollapragada and
  Morton}{1999}]{Bollapragada1999}
Bollapragada, S. and T.~Morton (1999).
\newblock A simple heuristic for computing non-stationary (${s,S}$) policies.
\newblock {\em Operations Research\/}~{\em 47\/}(4), 576--584.

\bibitem[\protect\citeauthoryear{Chopra and Meindl}{Chopra and
  Meindl}{2007}]{Chopra2007}
Chopra, S. and P.~Meindl (2007).
\newblock {\em Supply Chain Management. Strategy, Planning \& Operation\/} (3rd
  ed.).
\newblock Prentice Hall.

\bibitem[\protect\citeauthoryear{Ehrhardt}{Ehrhardt}{1979}]{Ehrhardt1979}
Ehrhardt, R. (1979).
\newblock Power approximation for computing (${s,S}$) inventory policies.
\newblock {\em Management Science\/}~{\em 25\/}(8), 777--786.

\bibitem[\protect\citeauthoryear{Federgruen and Zipkin}{Federgruen and
  Zipkin}{1984}]{Federgruen1984}
Federgruen, A. and P.~Zipkin (1984).
\newblock An efficient algorithm for computing optimal (${s, S}$) policies.
\newblock {\em Operations Research\/}~{\em 32\/}(6), 1268--1285.

\bibitem[\protect\citeauthoryear{Feng and Xiao}{Feng and Xiao}{2000}]{Feng2000}
Feng, Y. and B.~Xiao (2000).
\newblock A new algorithm for computing optimal $({s, S})$ policies in a
  stochastic single item/location inventory system.
\newblock {\em IIE Transactions\/}~{\em 32\/}(11), 1081--1090.

\bibitem[\protect\citeauthoryear{Freeland and Porteus}{Freeland and
  Porteus}{1980}]{Freeland1980}
Freeland, J.~R. and E.~L. Porteus (1980).
\newblock Evaluating the effectiveness of a new method for computing
  approximately optimal {$(s,S)$} inventory policies.
\newblock {\em Operations Research\/}~{\em 28\/}(2), 353--364.

\bibitem[\protect\citeauthoryear{Graves and Willems}{Graves and
  Willems}{2008}]{Graves2008}
Graves, S.~C. and S.~P. Willems (2008).
\newblock Strategic inventory placement in supply chains: {N}onstationary
  demand.
\newblock {\em Manufacturing \& Service Operations Management\/}~{\em 10\/}(2),
  278--287.

\bibitem[\protect\citeauthoryear{Guan and Miller}{Guan and
  Miller}{2008}]{Guan2008}
Guan, Y. and A.~J. Miller (2008).
\newblock Polynomial-time algorithms for stochastic uncapacitated lot-sizing
  problems.
\newblock {\em Operations Research\/}~{\em 56\/}(5), 1172--1183.

\bibitem[\protect\citeauthoryear{Iglehart}{Iglehart}{1963a}]{Iglehart1963a}
Iglehart, D.~L. (1963a).
\newblock Dynamic programming and stationary analysis of inventory problems.
\newblock In H.~Scarf, D.~Gilford, and M.~Shelly (Eds.), {\em Multistage
  inventory models and techniques}, pp.\  1--31. Stanford, CA: Stanford
  University Press.

\bibitem[\protect\citeauthoryear{Iglehart}{Iglehart}{1963b}]{Iglehart1963}
Iglehart, D.~L. (1963b).
\newblock Optimality of (${s, S}$) policies in the infinite horizon dynamic
  inventory problem.
\newblock {\em Management Science\/}~{\em 9\/}(2), 259--267.

\bibitem[\protect\citeauthoryear{Johnson}{Johnson}{1968}]{Johnson1968}
Johnson, E.~L. (1968).
\newblock On ${(s, S)}$ policies.
\newblock {\em Management Science\/}~{\em 15\/}(1), 80--101.

\bibitem[\protect\citeauthoryear{Karlin}{Karlin}{1960}]{Karlin1960}
Karlin, S. (1960).
\newblock Dynamic inventory policy with varying stochastic demands.
\newblock {\em Management Science\/}~{\em 6\/}(3), 231 -- 258.

\bibitem[\protect\citeauthoryear{Kurawarwala and Matsuo}{Kurawarwala and
  Matsuo}{1996}]{Kurawarwala1996}
Kurawarwala, A.~A. and H.~Matsuo (1996).
\newblock Forecasting and inventory management of short life-cycle products.
\newblock {\em Operations Research\/}~{\em 44\/}(1), 131--150.

\bibitem[\protect\citeauthoryear{Liu and Song}{Liu and Song}{2012}]{Liu2012}
Liu, F. and J.-S. Song (2012).
\newblock Good and bad news about the {$(S,T)$} policy.
\newblock {\em Manufacturing \& Service Operations Management\/}~{\em 14\/}(1),
  42--49.

\bibitem[\protect\citeauthoryear{Lulli and Sen}{Lulli and
  Sen}{2004}]{Lulli2004}
Lulli, G. and S.~Sen (2004).
\newblock A branch-and-price algorithm for multistage stochastic integer
  programming with application to stochastic batch-sizing problems.
\newblock {\em Management Science\/}~{\em 50\/}(6), 786.

\bibitem[\protect\citeauthoryear{Neale and Willems}{Neale and
  Willems}{2009}]{Neale2009}
Neale, J.~J. and S.~P. Willems (2009).
\newblock Managing inventory in supply chains with nonstationary demand.
\newblock {\em Interfaces\/}~{\em 39\/}(5), 388--399.

\bibitem[\protect\citeauthoryear{Porteus}{Porteus}{1979}]{Porteus1979}
Porteus, E.~L. (1979).
\newblock An adjustment to the {Norman-White} approach to approximating dynamic
  programs.
\newblock {\em Operations Research\/}~{\em 27\/}(6), 1203--1208.

\bibitem[\protect\citeauthoryear{Rao}{Rao}{2003}]{Rao2003}
Rao, U.~S. (2003).
\newblock Properties of the periodic review {$(R,T)$} inventory control policy
  for stationary, stochastic demand.
\newblock {\em Manufacturing \& Service Operations Management\/}~{\em 5\/}(1),
  37--53.

\bibitem[\protect\citeauthoryear{Scarf}{Scarf}{1960}]{Scarf1960}
Scarf, H.~E. (1960).
\newblock {O}ptimality of (${S,s}$) policies in the dynamic inventory problem.
\newblock In K.~J. Arrow, S.~Karlin, and P.~Suppes (Eds.), {\em Mathematical
  Methods in the Social Sciences}, pp.\  196--202. Stanford, CA: Stanford
  University Press.

\bibitem[\protect\citeauthoryear{Silver}{Silver}{1978}]{Silver1978}
Silver, E.~A. (1978).
\newblock Inventory control under a probabilistic time-varying, demand pattern.
\newblock {\em IIE Transactions\/}~{\em 10\/}(4), 371--379.

\bibitem[\protect\citeauthoryear{Silver}{Silver}{1981}]{Silver1981}
Silver, E.~A. (1981).
\newblock Operations research in inventory management: {A} review and critique.
\newblock {\em Operations Research\/}~{\em 29\/}(4), 628--645.

\bibitem[\protect\citeauthoryear{Silver}{Silver}{2008}]{Silver2008}
Silver, E.~A. (2008).
\newblock Inventory management: {A}n overview, {C}anadian publications,
  practical applications and suggestions for future research.
\newblock {\em INFOR\/}~{\em 46}, 15 -- 28.

\bibitem[\protect\citeauthoryear{Simchi-Levi, Kaminsky, and
  Simchi-Levi}{Simchi-Levi et~al.}{2003}]{Simchi-Levi2003}
Simchi-Levi, D., P.~Kaminsky, and E.~Simchi-Levi (2003).
\newblock {\em Designing and Managing the Supply Chain--Concepts, Strategies
  and Case Studies\/} (2nd ed.).
\newblock McGraw Hill.

\bibitem[\protect\citeauthoryear{Sivazlian}{Sivazlian}{1971}]{Sivazlian1971}
Sivazlian, B. (1971).
\newblock Dimensional and computational analysis in stationary {$(s, S)$}
  inventory problems with gamma distributed demand.
\newblock {\em Management Science\/}~{\em 17\/}(6), 307--311.

\bibitem[\protect\citeauthoryear{Veinott and Wagner}{Veinott and
  Wagner}{1965}]{Veinott1965}
Veinott, A.~F. and H.~M. Wagner (1965).
\newblock Computing optimal (${s,S}$) inventory policies.
\newblock {\em Management Science\/}~{\em 11\/}(5), 525--552.

\bibitem[\protect\citeauthoryear{Wagner, O'Hagan, and Lundh}{Wagner
  et~al.}{1965}]{Wagner1965}
Wagner, H.~M., M.~O'Hagan, and B.~Lundh (1965).
\newblock An empirical study of exactly and approximately optimal inventory
  policies.
\newblock {\em Management Science\/}~{\em 11\/}(7), 690--723.

\bibitem[\protect\citeauthoryear{Xiang, Rossi, Martin-Barragan, and
  Tarim}{Xiang et~al.}{2018}]{Xiang2018}
Xiang, M., R.~Rossi, B.~Martin-Barragan, and S.~A. Tarim (2018).
\newblock Computing non-stationary {$(s,S)$} policies using mixed integer
  linear programming.
\newblock {\em European Journal of Operational Research\/}~{\em 271\/}(2),
  490--500.

\bibitem[\protect\citeauthoryear{Zheng}{Zheng}{1991}]{Zheng1991a}
Zheng, Y.-S. (1991).
\newblock A simple proof for optimality of {$(s,S)$} policies in
  infinite-horizon inventory systems.
\newblock {\em Journal of Applied Probability\/}~{\em 28\/}(4), 802--810.

\bibitem[\protect\citeauthoryear{Zheng and Federgruen}{Zheng and
  Federgruen}{1991}]{Zheng1991}
Zheng, Y.~S. and A.~Federgruen (1991).
\newblock Finding optimal (${s,S}$) policies is about as simple as evaluating a
  single policy.
\newblock {\em Operations Research\/}~{\em 39\/}(4), 654--665.

\end{thebibliography}

\end{document}